\documentclass[preprint,12pt]{elsarticle}

\usepackage{amssymb}
\usepackage{amsthm}

\usepackage{mathrsfs}
\usepackage{amsfonts}

\usepackage{amsmath}
\usepackage{setspace}
\usepackage{indentfirst}
\usepackage{enumerate, hyperref}
\usepackage{dcolumn}

\newtheorem{thm}{Theorem}[section]
\newtheorem{lem}[thm]{Lemma}

\newtheorem{cor}[thm]{Corollary}
\theoremstyle{definition}
\newtheorem{rem}[]{Remark}[section]
\newcommand\R{{\mathbb R}}
\numberwithin{equation}{section}
\allowdisplaybreaks
\begin{document}

\begin{frontmatter}



\title{Regularity of 3D axisymmetric Navier-Stokes equations}

\author{Hui Chen}
\ead{poorpaul@zju.edu.cn}
\author{Daoyuan Fang}
\ead{dyf@zju.edu.cn}
\author{Ting Zhang}
\ead{zhangting79@zju.edu.cn}

\address{Department of Mathematics, Zhejiang University, Hangzhou 310027, China}

\begin{abstract}
In this paper, we study the three-dimensional axisymmetric Navier-Stokes system with nonzero swirl. By establishing a new key inequality for the pair $(\frac{\omega^{r}}{r},\frac{\omega^{\theta}}{r})$, we get several Prodi-Serrin type regularity criteria based on the angular velocity, $u^\theta$. Moreover, we obtain the global well-posedness result if the initial angular velocity $u_{0}^{\theta}$ is appropriate small in the critical space $L^{3}(\R^{3})$. Furthermore, we also get several Prodi-Serrin type regularity criteria based on one component of the solutions, say $\omega^3$ or $u^3$.
\end{abstract}


\end{frontmatter}


\section{Introduction}
Consider the initial value problem of 3D Navier-Stokes equations:
\begin{equation}
\left\{
\begin{array}{l}\label{1.1}
\partial_{t}\mathbf{u}+(\mathbf{u}\cdot\nabla)\mathbf{u}-\Delta \mathbf{u}+\nabla p=\mathbf{0},~(t,x)\in\R^+\times\R^{3}~,\\
\nabla\cdot \mathbf{u}=0~,\\
\mathbf{u}|_{t=0}=\mathbf{u_{0}}~.
\end{array}
\right.
\end{equation}
where $\mathbf{u}(t,x)=(u^{1},u^{2},u^{3})$, $p(t,x)$ and $\mathbf{u_{0}}$ denote, respectively, the fluid velocity field,  the pressure, and  the given initial velocity field.

For given $\mathbf{u_{0}}\in L^{2}(\R^{3})$ with $\mathrm{div}~\mathbf{u_{0}}=0$ in the sense of distribution, a global weak solution $\mathbf{u}$ to the Navier-Stokes equations was
constructed by Leray \cite{Leray} and Hopf \cite{Hopf}, which is called Leray-Hopf weak solution. Regularity of such Leray-Hopf weak solution in three dimension is one of the most outstanding open problems in the mathematical fluid mechanics.

Researchers are interested in the classical problem of finding sufficient conditions for the weak solutions such that they become regular.
The important result is usually referred as Prodi-Serrin (P-S) conditions (see \cite{Escauriaza,Fabes,Giga,Prodi,Serrin,Struwe,Takahashi}), i.e. if additional the weak solution $\mathbf{u}$  belongs to $L^{p,q}_{T}$, where $\frac{2}{p}+\frac{3}{q}\leq 1$, $3\leq q\leq\infty$,
then the weak solution becomes regular. Hugo Beir\~ao da Veiga \cite{Hugo} proved a  P-S type result with two components of $\mathbf{u}$. For the one component case,  Y. Zhou and Pokorn\'{y} \cite{Zhou} obtained the regularity criterion by imposing the integrability of  single component $u^{3}$ of the velocity field. Furthermore, Cao and Titi \cite{Titi} established the regularity criterion involving only one entry of the velocity gradient tensor, likely $\partial_{3}u^{3}$. However, the integral condition here is not optimal in the sense of scaling considerations. For such considerations, people also work on the regularity criteria involving one component of $\mathbf{u}$ and one other component, say velocity gradient tensor, likely \cite{Patrick}. Recently, the second author and Chenyin Qian \cite{CY.Qian1,CY.Qian2,CY.Qian3} developed the arguments above. They \cite{CY.Qian3} got several almost critical regularity conditions such that the weak
solutions of the 3D Navier-Stokes equations become regular, based on one component of the solutions,
say $u^3$ and $\partial_3u_3$.\\

Considering the  axisymmetric Navier-Stokes equations, there is a scaling invariant quantity $\|ru^{\theta}\|_{L^{\infty,\infty}_{T}}$ if $ru_{0}^{\theta}\in L^{\infty}(\R^{3})$, see \cite{Pokorny1, JG.Liu1} etc. So, it is very important to consider the critical regularity conditions for $u^\theta$ for the  axisymmetric Navier-Stokes equations.  Here, we assume that a solution $\mathbf{u}$ of the system (\ref{1.1}) of the form
$$
\mathbf{u}(t,x)=u^{r}(t,r,x_{3})\mathbf{e_{r}}+u^{\theta}(t,r,x_{3})\boldsymbol{e_{\theta}}+u^{3}(t,r,x_{3})\mathbf{e_{3}},
$$
where
\begin{equation*}
\mathbf{e_{r}}=(\frac{x_{1}}{r},\frac{x_{2}}{r},0),~\boldsymbol{e_{\theta}}=(-\frac{x_{2}}{r},\frac{x_{1}}{r},0),~\mathbf{e_{3}}=(0,0,1),
\ r=\sqrt{x_{1}^{2}+x_{2}^{2}}.
\end{equation*}
In the above, $u^{\theta}$ is called the angular velocity. For the axisymmetric solutions of Navier-stokes system, we can equivalently reformulate (\ref{1.1}) as
\begin{equation}\label{1.2}
\left\{
\begin{array}{l}
\frac{\tilde{D}}{Dt}u^{r}-(\partial_{r}^{2}+\partial_{3}^{2}+\frac{1}{r}\partial_{r}-\frac{1}{r^{2}})u^{r}-\frac{(u^{\theta})^{2}}{r}+\partial_{r}p=0,\\
\frac{\tilde{D}}{Dt}u^{\theta}-(\partial_{r}^{2}+\partial_{3}^{2}+\frac{1}{r}\partial_{r}-\frac{1}{r^{2}})u^{\theta}+\frac{u^{\theta}u^{r}}{r}=0,\\
\frac{\tilde{D}}{Dt}u^{3}-(\partial_{r}^{2}+\partial_{3}^{2}+\frac{1}{r}\partial_{r})u^{3}+\partial_{3}p=0,\\
\partial_{r}u^{r}+\frac{1}{r}u^{r}+\partial_{3}u^{3}=0,\\
(u^{r},u^{\theta},u^{3})|_{t=0}=(u_{0}^{r},u_{0}^{\theta},u_{0}^{3}),
\end{array}
\right.
\end{equation}
where we denote the convection derivative $\frac{\tilde{D}}{Dt}$ as
$$
\frac{\tilde{D}}{Dt}=\partial_{t}+u^{r}\partial_{r}+u^{3}\partial_{3}.
$$

For the axisymmetric velocity field $\mathbf{u}$, we can also compute the vorticity $\boldsymbol{\omega}=\mathrm{curl}~\mathbf{u}$ as follows,
$$
\boldsymbol{\omega}=\omega^{r}\mathbf{e_{r}}+\omega^{\theta}\boldsymbol{e_{\theta}}+\omega^{3}\mathbf{e_{3}}
$$
with $
\omega^{r}=-\partial_{3}u^{\theta},~\omega^{\theta}=\partial_{3}u^{r}-\partial_{r}u_{3},~\omega^{3}=\partial_{r}u^{\theta}+\frac{u^{\theta}}{r}$. Furthermore, $(\omega^{r},\omega^{\theta},\omega^{3})$ satisfies
\begin{equation}\label{1.3}
\left\{
\begin{array}{l}
\frac{\tilde{D}}{Dt}\omega^{r}-(\partial_{r}^{2}+\partial_{3}^{2}+\frac{1}{r}\partial_{r}-\frac{1}{r^{2}})\omega^{r}-(\omega^{r}\partial_{r}+\omega^{3}\partial_{3})u^{r}=0,\\
\frac{\tilde{D}}{Dt}\omega^{\theta}-(\partial_{r}^{2}+\partial_{3}^{2}+\frac{1}{r}\partial_{r}-\frac{1}{r^{2}})\omega^{\theta}-\frac{2u^{\theta}\partial_{3}u^{\theta}}{r}-\frac{u^{r}\omega^{\theta}}{r}=0,\\
\frac{\tilde{D}}{Dt}\omega^{3}-(\partial_{r}^{2}+\partial_{3}^{2}+\frac{1}{r}\partial_{r})\omega^{3}-(\omega^{r}\partial_{r}+\omega^{3}\partial_{3})u^{3}=0,\\
(\omega^{r},\omega^{\theta},\omega^{3})|_{t=0}=(\omega_{0}^{r},\omega_{0}^{\theta},\omega_{0}^{3}),
\end{array}
\right.
\end{equation}
Let $\mathbf{b}=u^{r}\mathbf{e_{r}}+u^{3}\mathbf{e_{3}}$. Then, we have that
\begin{equation}
\mathrm{div}~\mathbf{b}=0~~\text{and}~~\mathrm{curl}~\mathbf{b}=\omega^{\theta}\boldsymbol{e_{\theta}}.\label{1.4}
\end{equation}
Therefore, $(\Theta,\Gamma)=(ru^{\theta},\frac{\omega^{\theta}}{r})$ satisfy
\begin{equation}
\left\{\begin{array}{l}
(\partial_{t}+\mathbf{b}\cdot \nabla-\Delta+\frac{2}{r}\partial_{r})\Theta=0\\
(\partial_{t}+\mathbf{b}\cdot \nabla-\Delta-\frac{2}{r}\partial_{r})\Gamma-\partial_{3}(\frac{u^{\theta}}{r})^{2}=0
\end{array}
\right.
\end{equation}
This is the key ingredient for us to study the regularity criteria of the axisymmetric Navier-Stokes equations. We recall that global well-posedness result was firstly proved under no swirl assumption, i.e. $u^{\theta}=0$, independently by Ukhovskii and Yudovich \cite{Ukhovskii} , and Ladyzhenskaya \cite{Ladyzhenskaya}, also \cite{Leonardi} for a refined proof.\\

When the angular velocity $u^{\theta}$ is not trivial, the global well-posedness problem is still open. Recently, tremendous efforts and interesting progress have been made on the regularity problem of the axisymmetric Navier-Stokes equations\cite{Chae,Chen1,Chen2,ZF.Z,Koch,Pokorny2,Pokorny3,Z.Lei,Pokorny1,P.Zhang}. P. Zhang and the third author \cite{P.Zhang} investigated the global well-posedness with various types of smallness conditions on the initial angular velocity $u_{0}^{\theta}$ of the initial velocity field. In \cite{Chen1,Chen2}, Chen, Strain, Tsai and Yau  proved that the suitable weak solutions are smooth if the velocity field $\mathbf{u}$ satisfies $r|\mathbf{u}|\leq C<\infty$. Applying the Liouville type theorem for the ancient solutions of Navier-Stokes equations, Zhen Lei and Qi S. Zhang \cite{Z.Lei} obtained the similar result $\mathbf{b}\in L^\infty_T(BMO^{-1})$.
Moreover, there are many significant results under the sufficient condition for regularity of axially symmetric solution  of type
$\omega^{\theta}\in L^{p,q}_{T}$, $\frac{2}{p}+\frac{3}{q}\leq2,\frac{3}{2}<q<\infty$ in \cite{Chae}, and $\omega^{\theta}\in L^{1}((0,T);\dot{B}_{\infty,\infty}^{0})$ in \cite{ZF.Z}. And concerning on one velocity component, it has been shown that the axisymmetric solution is smooth in $(0,T)\times\R^3$ when $r^d(u^{r})^-\in L^{p,q}_{T}$ with $(d,p,q)\in\{(-1,1)\times(1,\infty)\times(\frac{3}{2},\infty),\frac{2}{p}+\frac{3}{q}\leq1-d\}$ or
$\{(-1,1)\times\{\infty\}\times(\frac{3}{2},\infty], \frac{3}{q}<1-d\}$ or $\{\{-1\}\times(1,\infty)\times(\frac{3}{2},\infty),\frac{2}{p}+\frac{3}{q}\leq2\}$ in \cite{Pokorny3}. In particular, the authors in \cite{Pokorny3} also considered the conditions on $u^{\theta}$,  $r^du^{\theta}\in L^{p,q}_{T}$ with $(d,p,q)\in\{[0,\frac{s-4}{2s})\times(4,\infty]\times(2,\infty],\frac{2}{p}+\frac{3}{q}<1-d\}$ or $\{(-\infty,\frac{5}{6})\times\{\infty\}\times\{\infty\}\}$.

The following  regularity criteria of $u^\theta$  greatly develop the corresponding regularity criteria in \cite{Pokorny2,Pokorny3,Pokorny1,P.Zhang}.

\begin{thm}\label{thm1.2}
Let $\mathbf{u}$ be an axisymmetric weak solution of the Navier-Stokes equations (\ref{1.1}) with the axisymmetric initial data $\mathbf{u_{0}}\in H^{2}(\R^{3})$ and $\mathrm{div}~\mathbf{u_{0}}=0$. If one of following conditions holds true
\begin{enumerate}
  \item[(1)] $r^{d} u^{\theta} \in L^{p,q}_{T}$,
where $\frac{2}{p}+\frac{3}{q}\leq1-d$,   $\frac{3}{1-d} < q\leq \infty$, $\frac{2}{1-d} \leq p\leq\infty$,
  \item[(2)] $r^{d} u^{\theta} \in L^{\infty,\frac{3}{1-d}}_{T}$, and there exist  $\alpha>0$ and sufficient small $\varepsilon>0$  such that
  $$\|r^{d} u^{\theta} 1_{r\leq\alpha} \|_{L^{\infty,\frac{3}{1-d}}_{T}}\leq \varepsilon,$$
\end{enumerate}
 where $0\leq d<1$,
then $\mathbf{u}$ is smooth in $(0,T]\times \R^{3}$.
\end{thm}

To prove it, the key point is that we find the pair $(\Phi,\Gamma)=(\frac{\omega^{r}}{r},\frac{\omega^{\theta}}{r})$ satisfying the following equations
 \begin{equation}\label{1.6}
\left\{
\begin{array}{lcl}
\partial_{t}\Phi+(\mathbf{b}\cdot\nabla)\Phi-(\Delta+\frac{2}{r}\partial_{r})\Phi-(\omega^{r}\partial_{r}+\omega^{3}\partial_{3})\frac{u^{r}}{r}=0,\\
\partial_{t}\Gamma+(\mathbf{b}\cdot\nabla)\Gamma-(\Delta+\frac{2}{r}\partial_{r})\Gamma+2\frac{u^{\theta}}{r}\Phi=0,
\end{array}
\right.
\end{equation}
 and the  following identities:
\begin{equation}\label{j1}
\frac{1}{2}\frac{d}{dt}\Phi^{2}+\frac{1}{2}(\mathbf{b}\cdot\nabla)\Phi^{2}+ (\Delta+\frac{2}{r}\partial_{r})\Phi\cdot \Phi=(\omega^{r}\partial_{r}+\omega^{3}\partial_{3})\frac{u^{r}}{r} \Phi,
\end{equation}
\begin{equation}\label{j2}
\frac{1}{2}\frac{d}{dt}\Gamma^{2}+\frac{1}{2}(\mathbf{b}\cdot\nabla)\Gamma^{2}+(\Delta+\frac{2}{r}\partial_{r})\Gamma \cdot\Gamma
=-2\frac{u^{\theta}}{r}\Gamma \Phi.
\end{equation}
 In fact, Thomas Y. Hou and Congming Li \cite{Thomas} introduced a 1-dimensional model that first-order approximates the Navier-Stokes equations, which just corresponds to the zero right hand side of above identities.
 In additional, the following two technical points  are also important: the first is
  that we give the explicit expression of $\frac{u^r}{r}$ by $\Gamma$,
 see  Lemma \ref{lem2.4};
    the second is that we establish a general Sobolev-Hardy inequality,
 see  Lemma \ref{lem2.1}.
 By above techniques, one reduces the problems to
estimating certain terms of particular forms. For example,
  we need to control the term $
I_{1}=\int_{\R^{3}}\left|u^{\theta}\partial_{r}\frac{u^{r}}{r}\partial_{3}\Phi\right|~dx$ and others, see (\ref{3.2})-(\ref{3.3}).
From  H\"{o}lder's inequality, we have
$$
I_{1}
\leq\|r^{d}u^{\theta}\|_{q} ~\| \frac{\partial_{r}\frac{u^{r}}{r}}{r^{d}}~\|_{\frac{2q}{q-2}} ~\|\partial_{3}\Phi\|_{2}.
$$
Using the general Sobolev-Hardy inequality in  Lemma \ref{lem2.1}, we can bound $\| \frac{\partial_{r}\frac{u^{r}}{r}}{r^{d}}~\|_{\frac{2q}{q-2}}$ by $\|\tilde{\nabla}^k\frac{u^r}{r}\|_2$, $k=1$, $2$. Combining the explicit expression of $\frac{u^r}{r}$ by $\Gamma$ in  Lemma \ref{lem2.4}, one can bound $\| \frac{\partial_{r}\frac{u^{r}}{r}}{r^{d}}~\|_{\frac{2q}{q-2}}$ by $\|\partial_3^{(k-1)}\Gamma\|_2$ (please see the details in (\ref{3.5})).

\begin{rem}
  Our method fails in the case $d=1$. The key point is that one cannot bounded $\|\frac{f}{r}\|_{2}$ by $\|f\|_{H^1}$.
  In Theorem \ref{thm1.2}, we add a small assumption in the case $r^{d} u^{\theta} \in L^{\infty,\frac{3}{1-d}}_{T}$. The reason is that we cannot using Gronwall's inequality in the case $r^{d} u^{\theta} \in L^{\infty,\frac{3}{1-d}}_{T}$. To overcome this small assumption, one possible method should be the De Giorgi type argument or Nash type method (like \cite{Z.Lei}). But, using the De Giorgi type argument or Nash type method, one needs some additional condition on $\mathbf{b}$ to control the convection terms. In our proof, we just use the energy method and divergence free condition to estimate the system (\ref{1.6}), and the convection terms are not trouble (see (\ref{3.2})).
\end{rem}

\begin{rem} Theorem \ref{thm1.2} tells us that if $ru^{\theta}$ is  H\"{o}lder continuity at the variable $r$ uniformly, i.e. there exist a $\alpha>0$ and constant $C$, such that
$$r|u^{\theta}|\leq C r^{\alpha},~ a.e.~(t,x)\in (0,T)\times\R^{3}$$
then $\mathbf{u}$ is regular. And the authors in \cite{Chen1, Chen2, Z.Lei} drawn a similar argument by using the Nash-Moser method.
\end{rem}

\begin{rem}  Theorem \ref{thm1.2} implies that
if the angular velocity $u^{\theta}$ satisfies the Prodi-Serrin condition
    $$
    u^{\theta}\in L^{p}((0,T);L^{q}(\R^{3})),\ \frac{2}{p}+\frac{3}{q}\leq1,\  3< q\leq\infty,
    $$
  then the solution $\mathbf{u}$ is smooth in $(0,T]\times\R^3$.
In \cite{P.Zhang}, P. Zhang and the third author  obtained one special regularity criterion $u^\theta\in L^{4,6}_T$. From Theorem \ref{thm1.2} and the interpolation theorem, we have that
if $ru^\theta_0\in L^\infty$ and the angular velocity $u^{\theta}$ satisfies
    $$
     \frac{u^{\theta}}{r^d}\in L^{p}((0,T);L^{q}(\R^{3})),\ \frac{2}{p}+\frac{3}{q}\leq 1+d,\  0<d\leq 1,\ \frac{3}{1+d}< q\leq\infty,
    $$
  then the solution $\mathbf{u}$ is smooth in $(0,T]\times\R^3$.
\end{rem}

Moreover, we will prove a global regularity theorem  only on that the initial angular velocity $u_{0}^{\theta}$ is appropriately small in the critical space $L^{3}(\R^{3})$.

\begin{thm}\label{thm1.4}
Assume that the axisymmetric initial data $\mathbf{u_{0}}\in H^{2}(\R^{3})$ and $\mathrm{div}~\mathbf{u_{0}}=0$. There exist positive constants $C_{0}$, $C_{1}$ and $C_{2}$, such that if the initial angular velocity $u_{0}^{\theta}$ satisfies
\begin{equation}\label{1.7}
\|r^du^{\theta}_0\|_{\frac{3}{1-d}} \exp(\mathscr{A}) \leq\frac{1}{2C_{0}},
\end{equation}
where $0\leq d<1$, $\mathscr{A}=\mathscr{A}\left(\|\mathbf{u_{0}}\|_{2},\|\omega_{0}^{\theta}\|_{2},\|\frac{\omega_{0}}{r}\|_{2},\|\partial_{3}\frac{u_{0}^{\theta}}{r}\|_{2} \right)$ is defined by
$$
\mathscr{A}=C_{2}\|\mathbf{u_{0}}\|_{2}^{2}~(\|\omega_{0}^{\theta}\|_{2}^{2}+C_{1}(\|\frac{\omega_{0}}{r}\|_{2}+\|\partial_{3}\frac{u_{0}^{\theta}}{r}\|_{2})^{\frac{4}{3}}\|\mathbf{u_{0}}\|_{2}^{2}),
$$
then the system (\ref{1.1}) has a unique global strong solution $\mathbf{u}\in C(\R^+;H^{2}(\R^{3}))\cap L^{2}_{loc}(\R^+;\dot{H}^{3}(\R^{3}))$.
\end{thm}

We also respective investigate the regularity criterion  of the vorticity component $\omega^{3}$ or  the velocity component $u^{3}$ by introducing the potential $\psi$ defined in \cite{JG.Liu1}.
\begin{thm}\label{thm1.1}
Let $\mathbf{u}$ be an axisymmetric weak solution of the Navier-Stokes equations  (\ref{1.1}) with the axisymmetric initial data $\mathbf{u_{0}}\in H^{2}(\R^{3})$ and $\mathrm{div}~\mathbf{u_{0}}=0$. If $ru_{0}^{\theta}\in L^{\infty}(\R^{3})$  and $\omega^{3}\in L^{p,q}_{T}$, $\frac{2}{p}+\frac{3}{q}\leq2$, $\frac{3}{2} < q<\infty$, $1 < p\leq\infty$,
then the weak solution $\mathbf{u}$ is smooth in $(0,T]\times \R^{3}$.
\end{thm}

\begin{thm}\label{thm1.5}
Let $\mathbf{u}$ be an axisymmetric weak solution of the Navier-Stokes equations (\ref{1.1}) with the axisymmetric initial data $\mathbf{u_{0}}\in H^{2}(\R^{3})$ and $\mathrm{div}~\mathbf{u_{0}}=0$. If $u^{3}$ satisfies one of following conditions,
\begin{enumerate}
\item[(1)]  $r^{d} u^{3} \in L^{p,q}_{T}$, $\frac{2}{p}+\frac{3}{q}\leq1-d$,   $\frac{3}{1-d} < q\leq\infty$, $\frac{2}{1-d} \leq p\leq\infty$,
\item[(2)]   $r^{d} u^{3} \in L^{\infty,\frac{3}{1-d}}_{T}$, and there exist  $\beta>0$ and sufficient small $\varepsilon_1>0$  such that
  $$\|r^{d} u^{3} 1_{r\leq\beta} \|_{L^{\infty,\frac{3}{1-d}}_{T}}\leq \varepsilon_1,$$
\end{enumerate}
where $0\leq d<1$,
then $\mathbf{u}$  is smooth in $(0,T]\times \R^{3}$.
\end{thm}

Since $\psi$ is bounded if $r|u^{3}|\leq C$, we deduce following corollary from \cite{Z.Lei}.
\begin{cor}\label{cor1.6}
Let $\mathbf{u}$ be an axisymmetric  suitable weak solution of the Navier-Stokes equations (\ref{1.1}) with the axisymmetric initial data $\mathbf{u_{0}}\in L^{2}(\R^{3})$, $\mathrm{div}~\mathbf{u_{0}}=0$, and $ru_{0}^{\theta}\in L^{\infty}(\R^{3})$. Suppose $ru^{3}\in L^{\infty,\infty}_T$, then $\mathbf{u}$  is smooth in $(0,T]\times \R^{3}$.
\end{cor}

In Section \ref{S2}, we establish some important lemmas for the use of proof, like the explicit expression of $\frac{u^r}{r}$ by $\frac{\omega^\theta}{r}$ in  Lemma \ref{lem2.4},
a general Sobolev-Hardy inequality in  Lemma \ref{lem2.1}. Then we will prove the regularity criteria by $u^{\theta}$, $\omega^3$ and $u^{3}$, respectively, in Section \ref{S3}. At the end, we prove the global well-posedness result in Theorem \ref{thm1.4} in Section \ref{S4}.\\

\textbf{Notations.} We introduce the Banach space $L^{p,q}_{T}$, equipped with norm
$$
\|f\|_{L^{p,q}_{T}}=\left\{
\begin{array}{lcl}
\bigl(\int_{0}^{T}~\|f(t)\|^{p}_{q}~dt\bigr)^{\frac{1}{p}},&~ & \text{if}~1\leq p<\infty,\\
\text{ess sup}_{t\in(0,T)}\|f(t)\|_{q}, &~ & \text{if}~p=\infty,
\end{array}\right.
$$
where
$$
\|f(t)\|_{q}=\left\{\begin{array}{lcl}\bigl(\int_{\R^{3}}~|f(t,x)|^{q}~dx \bigr)^{\frac{1}{q}},&~ & \text{if}~1\leq q<\infty,\\
\text{ess sup}_{x\in\R^{3}}~|f(t,x)|, &~ & \text{if}~q=\infty.
\end{array}\right.
$$

\section{Preliminaries}\label{S2}
Now we begin with   the 1-D Hardy inequality in \cite{Hardy} (Theorem 330, page 245).
\begin{lem}\label{lem2.2}
If $q>1,\sigma\neq1$, and $f(r)$ is nonnegative measurable function, $F(r)$ is defined by
$$
F(r)=\int_{0}^{r}f(t)~dt~(\sigma>1),~~F(r)=\int_{r}^{\infty}f(t)~dt~(\sigma<1),
$$
then
$$
\int_{0}^{\infty}r^{-\sigma}F^{q}~dr<(\frac{q}{|\sigma-1|})^{q}\int_{0}^{\infty}r^{-\sigma}(rf)^{q}~dr
$$
\end{lem}

We will give some useful estimates in the axisymmetric Navier-Stokes equations, and refer to \cite{Pokorny1,JG.Liu1} for details. Note $\tilde\nabla=(\partial_{r},\partial_{3})$, and
$$
C_{s}(\overline{\R^{+}}\times\R)=\{f(r,z)\in C^{\infty}(\overline{\R^{+}}\times\R),~\partial_{r}^{2j}f(0^{+},z)=0,~j\geq0,j\in\mathbb{N}  \}.$$
\begin{lem}\label{lem2.3}
Assume $\mathbf{u}$ is the smooth axisymmetric solution of (\ref{1.1}) on $[0,T]$, with the initial data $\mathbf{u_{0}}$, and $\mathrm{curl} ~\mathbf{u}= \boldsymbol{\omega} $, then
\begin{enumerate}
\item[\textrm{i})]  $\mathbf{b}$, $\frac{u^{r}}{r}$, $\frac{u^{\theta}}{r}$ and $\frac{\omega^{\theta}}{r}$ are smooth with variables $(t,x)\in (0,T)\times \R^{3}$,

\item[\textrm{ii})] $\mathbf{u}=u^{\theta}\boldsymbol{e_{\theta}}+\nabla\times(\psi \boldsymbol{e_{\theta}})=-\partial_{3}\psi\boldsymbol{e_{r}}+u^{\theta}\boldsymbol{e_{\theta}}+\frac{\partial_{r}(r\psi)}{r}\boldsymbol{e_{3}}$, with
\begin{equation}
u^{\theta}(t,r,x_{3}), \ \
\psi(t,r,x_{3}), \ \
\omega^{\theta}(t,r,x_{3})\in C^{1}(0,T; C_{s}(\overline{\R^{+}}\times\R)),
\end{equation}

\item[\textrm{iii})] there exists a  constant $C=C(q)$, such that for $1<q<\infty$,
$$
\|\tilde\nabla u^{r}\|_{q}+\|\tilde\nabla u^{3}\|_{q}+\|\frac{u^{r}}{r}\|_{q}\leq C \|\omega^{\theta}\|_{q},
$$
    $$
\|\tilde\nabla u^{\theta}\|_{q}+\|\frac{u^{\theta}}{r}\|_{q}\leq C\|\nabla \mathbf{u}\|_{q} ,
$$
    $$
\|\tilde\nabla(\frac{u^{r}}{r})\|_{q}+\|\tilde\nabla(\frac{u^{\theta}}{r})\|_{q}  +
\|\tilde\nabla \omega^{r}\|_{q}+\|\tilde\nabla \omega^{3}\|_{q}+\|\frac{\omega^{r}}{r}\|_{q}+\|\tilde\nabla \omega^{\theta}\|_{q}+\|\frac{\omega^{\theta}}{r}\|_{q} \leq C \|\nabla^{2} \mathbf{u}\|_{q},
$$
    $$
\|\tilde\nabla(\frac{\omega^{\theta}}{r})\|_{q}   \leq C\|\nabla^{3} \mathbf{u}\|_{q},
$$

\item[\textrm{iv})] if in addition, $ru_{0}^{\theta}\in L^{\infty}(\R^{3})$, then also $ru^{\theta}\in L^{\infty,\infty}_{T}.$
\end{enumerate}
\end{lem}
Furthermore, we extend the following argument, which is widely used in various equations with free swirl (See \cite{Hammadi,CX.Miao}), to the axisymmetric solutions with nonzero swirl.
\begin{lem} \label{lem2.4}
Assume $\mathbf{u}$ is the smooth axisymmetric solution of (\ref{1.1}), and $\mathrm{curl} ~\mathbf{u}= \boldsymbol{\omega} $, then
\begin{equation}
\frac{u^{r}}{r}=\Delta^{-1}\partial_{3}(\frac{\omega^{\theta}}{r})-2\frac{\partial_{r}}{r}\Delta^{-2}\partial_{3}(\frac{\omega^{\theta}}{r}), \\
\end{equation}
here we have
\begin{equation}
\frac{\partial_{r}}{r}\Delta^{-1}W(r,x_{3})=\frac{x_{2}^{2}}{r^{2}}\mathcal{R}_{11}W+\frac{x_{1}^{2}}{r^{2}}\mathcal{R}_{22}W-2\frac{x_{1}x_{2}}{r^{2}}\mathcal{R}_{12}W
\end{equation}
with $\mathcal{R}_{ij}=\Delta^{-1}\partial_{x_{i}}\partial_{x_{j}}$.
Besides, we have, for $1<q<\infty$,
\begin{equation} \label{2.6}
\| \tilde\nabla\frac{u^{r}}{r}\|_{q}\leq C(q)~\|\frac{\omega^{\theta}}{r}\|_{q},
\end{equation}
\begin{equation} \label{2.6-1}
\|\tilde\nabla\tilde\nabla \frac{u^{r}}{r}\|_{q}\leq C(q)~\|\partial_{3} (\frac{\omega^{\theta}}{r})\|_{q}.
\end{equation}
\end{lem}
\begin{proof} Set $\Gamma=\frac{\omega^{\theta}}{r}$. From (\ref{1.4}), we have
\begin{equation*}
\Delta~\mathbf{b}=-\mathrm{curl}( \omega^{\theta}\boldsymbol{e_{\theta}})
                =\left(\partial_{3}(\omega^{\theta}\frac{x_{1}}{r}),\partial_{3}(\omega^{\theta}\frac{x_{2}}{r}),-\partial_{1}(\omega^{\theta}
                \frac{x_{1}}{r})-\partial_{2}(\omega^{\theta}\frac{x_{2}}{r})\right),
\end{equation*}
Since $\mathbf{b}=u^r(\frac{x_1}{r},\frac{x_2}{r},0)$, we have
    $$
  u^r=\sum_{i=1}^2\frac{x_i}{r}\partial_3\Delta^{-1}(x_i\Gamma)
  = r \partial_3\Delta^{-1} \Gamma-\sum_{i=1}^2\frac{x_i}{r}\partial_3[x_i,\Delta^{-1}]\Gamma.
    $$
Since
\begin{eqnarray*}
\Delta\left[x_{i},\Delta^{-1}\right]\Gamma    = 2\partial_{i}\Delta^{-1}\Gamma,
\end{eqnarray*}
then we get
$$
\left[x_{i},\Delta^{-1}\right]\Gamma= 2\partial_{i}\Delta^{-2}\Gamma= 2x_{i}\frac{\partial_{r}}{r}\Delta^{-2}\Gamma,$$
and
$$
\frac{u^{r}}{r}
               = \Delta^{-1}\partial_{3}\Gamma-2\frac{\partial_{r}}{r}\Delta^{-2}\partial_{3}\Gamma.
$$
In addition, using the polar coordinates $x_{1}=r\cos \theta,x_{2}=r\sin \theta $, we obtain
\begin{eqnarray*}
&&\frac{\partial_{r}}{r}\Delta^{-1}W(r,x_{3})= (\Delta-\partial_{r}^{2}-\partial_{3}^{2})\Delta^{-1}W(r,x_{3})\\
                                       &=& \sin^{2}\theta(\mathcal{R}_{11}W)(x)+\cos^{2}\theta(\mathcal{R}_{22}W)(x)
                                       -2\sin\theta\cos\theta(\mathcal{R}_{12}W)(x).
\end{eqnarray*}
Using the $L^{q}$-boundedness of Riesz operator, we easily obtain (\ref{2.6})-(\ref{2.6-1}).
\end{proof}

Then, we give a general Sobolev-Hardy inequality.  Badiale and Tarantello proved the case $q_*=\frac{q(N-s)}{N-q}$ in \cite{Marino} (Theorem 2.1). For the
convenience of reader's reading, we give the proof by another method.
\begin{lem}\label{lem2.1}
Set $\R^{N}=\R^{k}\times\R^{N-k}$ with $2\leq k\leq N$, and write $ x=(x',z)\in\R^{k}\times\R^{N-k} $. For a given real number $q$, $s$, such that $1<q<N,0\leq s\leq q$, and $s<k$, set $q_{*}\in[q,\frac{q(N-s)}{N-q}]$.  Then there exists a positive constant $C=C(s,q,N,k)$, such that for all $f\in C_{0}^{\infty}(\R^{N})$,
    \begin{equation}\label{2.7}
 \int_{\R^{N}} \frac{|f|^{q_{*}}}{|x'|^{s}}~dx \leq C \|f\|_q^{\frac{N-s}{q_*}-\frac{N}{q}+1}\|\nabla f\|_q^{\frac{N}{q}-\frac{N-s}{q_*}}.
    \end{equation}
In particular, we pick $N=3$, $k=2$, $q=2$, $q_{*}\in[2,2(3-s)]$, and assume $0\leq s<2$, $r=\sqrt{x_{1}^{2}+x_{2}^{2}}$. Then there exists a positive constant $C(s)$, such that for all $f\in C_{0}^{\infty}(\R^{3})$,
\begin{equation}\label{2.7-0}
 \left\|\frac{f}{r^{\frac{s}{q_{*}}}}\right\|_{q_{*}} \leq C_{q_*,s}\|f\|_2^{\frac{3-s}{q_*}-\frac{1}{2}}\|\nabla f\|_2^{\frac{3}{2}-\frac{3-s}{q_*}}.
 \end{equation}
\end{lem}
\begin{proof} By the Sobolev embedding Theorem, one obtain (\ref{2.7}) with $s=0$.
 When $q_*=q=s<k$, one easily obtain (\ref{2.7}) from the Hardy's inequality \cite{Garcia},
    \begin{equation}
        \int_{\mathbb{R}^N}|x'|^{-q}|f|^qdx\leq
     C   \int_{\mathbb{R}^{N-k}}\|\nabla_{x'}f\|_{L^q_{x'}}^qdz\leq C\|\nabla f\|_q^q.\label{2.11}
    \end{equation}
Then, we assume that $f\neq 0$, $0<s<q$. Set $q^*=\frac{Nq}{N-q}$.

\textbf{Case 1.}  $q_*=\frac{q(N-s)}{N-q}$.

\textbf{Case 1a.} $k=N$.

From
 H\"{o}lder's inequality and
the Hardy's inequality, we have
    \begin{equation}
       \int_{\R^{N}} \frac{|f|^{\frac{q(N-s)}{N-q}}}{|x|^{s}}~dx
       \leq C \left(
      \int_{\R^{N}} \frac{|f|^q}{|x|^{q}}~dx
      \right)^\frac{s}{q}\left(\int_{\R^{N}} f^{q^*}  dx
      \right)^{\frac{q-s}{q}}\leq C \|\nabla f\|_q.\label{2.10-0}
    \end{equation}

\textbf{Case 1b.} $k<N$.

Using
 H\"{o}lder's inequality and the Sobolev embedding Theorem, we obtain
        \begin{eqnarray*}
           \left|\int_{|x'|\geq \varepsilon} |x'|^{-s}|f|^{\frac{q(N-s)}{N-q}}dx'\right|^{\frac{N-q}{q(N-s)}}
         &   \leq& \left|\int_{|x'|\geq \varepsilon} |x'|^{-N} dx'\right|^{\frac{s(N-q)}{Nq(N-s)}}\|f\|_{L^{q^*}_{x'}}\\
        &  \leq& C \varepsilon^{(k-N)\frac{s(N-q)}{Nq(N-s)}} \| f\|_{L^{q^*}_{x'}},
        \end{eqnarray*}
    and
        \begin{eqnarray*}
          \left|\int_{|x'|< \varepsilon} |x'|^{-s}|f|^{\frac{q(N-s)}{N-q}}dx'\right|^{\frac{N-q}{q(N-s)}}
       &   \leq& \left(\int_{|x'|< \varepsilon}|x'|^{-\frac{s(N-q)}{q(N-s)}\gamma}dx' \right)^{\frac{1}{\gamma}}
       \|f\|_{L^{\beta}_{x'}}\\
        &  \leq& C_{s,q}\varepsilon^{\frac{k}{\gamma}-\frac{s(N-q)}{q(N-s)}}
        \|f\|_{L^{q^*}_{x'}}^{\theta}\|\nabla_{x'}f\|_{L^q_{x'}}^{1-\theta},
        \end{eqnarray*}
    where $\gamma$ is a constant satisfying $\frac{q(N-s)}{N-q}<\gamma<\frac{kq(N-s)}{s(N-q)}$ and $\frac{N-q}{q(N-s)}+\frac{1}{k}-\frac{1}{q}\geq \frac{1}{\gamma}\geq \frac{N-q}{q(N-s)}+\frac{1}{N}-\frac{1}{q}$, $\frac{1}{\gamma}+\frac{1}{\beta}=\frac{N-q}{q(N-s)}$ and
    $\theta=\frac{\frac{N-q}{q(N-s)}+\frac{1}{k}-\frac{1}{\gamma}-\frac{1}{q}}{\frac{1}{k}-\frac{1}{N}}$. Then choosing $\varepsilon=\left(\frac{\|f\|_{L^{q^*}_{x'}}}{\|\nabla_{x'}f\|_{L^q_{x'}}}\right)^\frac{N}{N-k}$, we get
        \begin{equation}
          \||x'|^{-\frac{s(N-q)}{q(N-s)}}f\|_{L^{q_*}_{x'}}
          \leq C_{s,q}
          \|f\|_{L^{q^*}_{x'}}^{1-\frac{s(N-q)}{q(N-s)}}
          \|\nabla_{x'}f\|_{L^q_{x'}}^{\frac{s(N-q)}{q(N-s)}}.\label{2.9}
        \end{equation}
 From  (\ref{2.9}), using
 H\"{o}lder's inequality,   we have
        \begin{equation}
         \||x'|^{-\frac{s(N-q)}{q(N-s)}}f\|_{{q_*}}
          \leq C_{s,q}
          \|f\|_{{q^*}}^{1-\frac{s(N-q)}{q(N-s)}}
          \|\nabla_{x'}f\|_{q}^{\frac{s(N-q)}{q(N-s)}}
      \leq C_{s,q}
          \|\nabla f\|_{q}.\label{2.10}
    \end{equation}

\textbf{Case 2.} $q_*=q$.

    Using
 H\"{o}lder's inequality and the Sobolev embedding Theorem, we obtain
    $$
           \left|\int_{|x'|\geq \varepsilon} |x'|^{-s}|f|^{q}dx'\right|^{\frac{1}{q}}
            \leq    \varepsilon^{-\frac{s}{q}} \|f\|_{L^q_{x'}},
      $$
and
        \begin{eqnarray*}
         \left|\int_{|x'|< \varepsilon} |x'|^{-s}|f|^{q}dx'\right|^{\frac{1}{q}}
       &   \leq& \left(\int_{|x'|< \varepsilon}|x'|^{-\frac{s}{q}\gamma}dx' \right)^{\frac{1}{\gamma}} \|f\|_{L^{\beta}_{x'}}\\
        &  \leq& C_{s,q}\varepsilon^{\frac{k}{\gamma}-\frac{s}{q}}\|f\|_{L^q_{x'}}^{1-\frac{k}{\gamma}}
         \|\nabla_{x'}f\|_{L^q_{x'}}^{\frac{k}{\gamma}},
        \end{eqnarray*}
    where $\gamma$ is a constant satisfying $q<\gamma<\frac{kq}{s}$ and $\gamma\geq k$, $\frac{1}{\gamma}+
    \frac{1}{\beta}= \frac{1}{q}$. Then choosing $\varepsilon=\frac{\|f\|_{L^q_{x'}}}{\|\nabla_{x'}f\|_{L^q_{x'}}}$, we get
\begin{equation}
          \||x'|^{-\frac{s}{q}}f\|_{L^{q}_{x'}}
          \leq C_{s,q}\|f\|_{L^q_{x'}}^{1-\frac{s}{q}}
          \|\nabla_{x'}f\|_{L^q_{x'}}^{\frac{s}{q}}.\label{2.9-0}
        \end{equation}
 From  (\ref{2.9-0}), using
 H\"{o}lder's inequality,   we have
        \begin{equation}
       \||x'|^{-\frac{s}{q}}f\|_{q}\leq C_{s,q}\|f\|_{q}^{1-\frac{s}{q}}
          \|\nabla_{x'}f\|_{q}^{\frac{s}{q}}.\label{2.12}
    \end{equation}
 From (\ref{2.11}), (\ref{2.10-0}), (\ref{2.10})  and (\ref{2.12}), by the interpolation theorem, we have that (\ref{2.7}) holds for all  $q_*\in [q,\frac{q(N-s)}{N-q}]$.
 This finishes the proof of (\ref{2.7}).
\end{proof}

From the results in \cite{Pokorny2,Pokorny1,P.Zhang}, one have the following lemma.
\begin{lem}
\label{lem2.5}
Let $\mathbf{u}\in C([0,T);H^2(\R^3))\cap L^2_{loc}([0,T);H^3(\R^3))$ be the unique axisymmetric solution of the Navier-Stokes equations with the axisymmetric initial data $\mathbf{u_{0}}\in H^{2}(\R^{3})$ and $\mathrm{div}~\mathbf{u_{0}}=0$.
If in addition, $T<\infty$ and $ \|\frac{u^{\theta}}{r}\|_{L^{4,4}_{T}} < \infty$, then $\mathbf{u}$ can be continued beyond $T$.
\end{lem}

\section{Proof of the regularity criteria}\label{S3}
It is well-known that: if the axisymmetric initial data $\mathbf{u_{0}}\in H^{2}(\R^{3})$ and $\mathrm{div}~\mathbf{u_{0}}=0$, we can construct a global axisymmetric weak solution $\mathbf{u}$ (see \cite{Cafferelli,Chae}), satisfying the energy inequality,
\begin{equation}\label{3.1}
\frac{1}{2}\|\mathbf{u}(t)\|^{2}_{2}+\int_{0}^{t}\|\nabla\mathbf{u}\|_{2}^{2}~d\tau \leq\frac{1}{2}\|\mathbf{u_{0}}\|_{2}^{2},~\textrm{ for all  } t\geq0.
\end{equation} Moreover, the system (\ref{1.1}) has a local unique solution $\mathbf{u}$ on $[0,T^{*})$, satisfying $\mathbf{u}\in C([0,T^{*});H^{2}(\R^{3}))\cap L^{2}_{loc}([0,T^{*});\dot{H}^{3}(\R^{3}))$,  where
$T^{*}$ is the earliest blow-up point  (see \cite{Constantin.P} for instance).

\begin{lem}\label{lem2.6}
Let $\mathbf{u}\in C([0,T);H^2(\R^3))\cap L^2_{loc}([0,T);H^3(\R^3))$ be the unique axisymmetric solution of the Navier-Stokes equations with the axisymmetric initial data $\mathbf{u_{0}}\in H^{2}(\R^{3})$ and $\mathrm{div}~\mathbf{u_{0}}=0$.
If in addition, $T<\infty$ and $\sup_{t\in[0,T)} \|\Gamma\|_{L^{\infty,2}_{t}} < \infty$, then $\mathbf{u}$ can be continued beyond $T$.
\end{lem}
\begin{proof}
  Multiplying the $u^{\theta}$ equation of (\ref{1.2})$_2$ by $\frac{(u^{\theta})^{3}}{r^{2}}$, and integrating the resulting equation over $\R^{3}$, applying (\ref{2.6}) and Cauchy-Schwarz inequality, we have for all $t\in[0,T)$
\begin{eqnarray*}
\frac{1}{4}\frac{d}{dt}\left\|\frac{(u^{\theta})^{2}}{r}\right\|_{2}^{2}+\frac{3}{4}\left\|\tilde\nabla\frac{(u^{\theta})^{2}}{r}\right\|_{2}^{2}
+\frac{3}{4}\left\|\frac{u^{\theta}}{r}\right\|_{4}^{4}&~=~& -\frac{3}{2}\int_{\R^{3}}\frac{u^{r}}{r} \frac{(u^{\theta})^{2}}{r}\frac{(u^{\theta})^{2}}{r}~dx \\
&~\leq~& C \left\|\frac{u^{r}}{r}\right\|_{6} ~\left\|\frac{(u^{\theta})^{2}}{r}\right\|_{2}~ \left\|\frac{(u^{\theta})^{2}}{r}\right\|_{3} \\
&~\leq~& C \left\|\tilde\nabla\frac{u^{r}}{r}\right\|_{2} ~\left\|\frac{(u^{\theta})^{2}}{r}\right\|_{2}^{\frac{3}{2}}~ \left\|\tilde\nabla\frac{(u^{\theta})^{2}}{r}\right\|_{2}^{\frac{1}{2}}\\
&~\leq~& C \|\Gamma\|_{2}^{\frac{4}{3}} ~\left\|\frac{(u^{\theta})^{2}}{r}\right\|_{2}^{2}+ \frac{1}{2}\left\|\tilde\nabla\frac{(u^{\theta})^{2}}{r}\right\|_{2}^{2}.
\end{eqnarray*}
 Applying Gronwall's inequality, using the fact that $\sup_{t\in[0,T)} \|\Gamma\|_{L^{\infty,2}_{t}} < \infty$ , we   obtain
$ \|\frac{u^{\theta}}{r}\|_{L^{4,4}_{T}}^{4} \leq C(\|\frac{u^{\theta}_0}{\sqrt{r}}\|_{4}^{4},T)$. Thus we prove this lemma by Lemma \ref{lem2.5}.
\end{proof}

\vspace{0.5cm}
\noindent\textbf{Proof of Theorem \ref{thm1.2}.}

Assume that $T^*\leq T$. From (\ref{j1}) and (\ref{j2}),  integrating them over $\R^3$ respectively, we
 obtain for all $t\in[0,T^*)$
\begin{eqnarray}
\frac{1}{2}\frac{d}{dt}\|\Phi\|_{2}^{2}+\|\tilde\nabla\Phi\|_{2}^{2}  &=& \int_{\R^{3}}(\omega^{r}\partial_{r}+\omega^{3}\partial_{3})\frac{u^{r}}{r}~\Phi ~dx+\int_{\R}\int_{0}^{\infty}\partial_{r}(\Phi)^{2}~drdx_{3}\nonumber\\
&\leq&\int_{\R^{3}}(\omega^{r}\partial_{r}+\omega^{3}\partial_{3})\frac{u^{r}}{r}~\Phi ~dx \nonumber\\
&=&2\pi \int_{\R}\int^\infty_0(-\partial_{3}u^{\theta}\partial_{r}\frac{u^{r}}{r}\Phi+\frac{\partial_{r}(ru^{\theta})}{r}\partial_{3} \frac{u^{r}}{r}\Phi)~rdrdx_{3}\nonumber\\
&=&\int_{\R^{3}}u^{\theta}(\partial_{3}\partial_{r}\frac{u^{r}}{r}\Phi+\partial_{r}\frac{u^{r}}{r}\partial_{3}\Phi)~dx
-\int_{\R^{3}}u^{\theta}(\partial_{r}\partial_{3} \frac{u^{r}}{r}\Phi+\partial_{3} \frac{u^{r}}{r}\partial_{r}\Phi)~dx\nonumber\\
&=& \int_{\R^{3}}u^{\theta}(\partial_{r}\frac{u^{r}}{r}\partial_{3}\Phi-\partial_{3} \frac{u^{r}}{r}\partial_{r}\Phi)~dx\nonumber\\
&\leq&I_{1}+I_{2}\label{3.2}
\end{eqnarray}
and
\begin{equation}\label{3.3}
\frac{1}{2}\frac{d}{dt}\|\Gamma\|_{2}^{2}+\|\tilde\nabla\Gamma\|_{2}^{2}
 \leq 2I_{3},
\end{equation}
where
\begin{equation*}
I_{1}=\int_{\R^{3}}\left|u^{\theta}\partial_{r}\frac{u^{r}}{r}\partial_{3}\Phi\right|~dx,\ \
I_{2}=\int_{\R^{3}}\left|u^{\theta}\partial_{3} \frac{u^{r}}{r}\partial_{r}\Phi\right|~dx,\ \
I_{3}=\int_{\R^{3}}\left|\frac{u^{\theta}}{r}~\Phi~\Gamma\right|~dx.
\end{equation*}
From (\ref{3.2}) and (\ref{3.3}), we get for all $t\in[0,T^*)$
\begin{equation}\label{3.4}
\frac{1}{2}\frac{d}{dt}(\|\Phi\|_{2}^{2}+\|\Gamma\|_{2}^{2})+\|\tilde\nabla\Phi\|_{2}^{2}+\|\tilde\nabla\Gamma\|_{2}^{2}\leq I_{1}+I_{2}+2I_{3}.
\end{equation}
Then, we estimate $I_{1},I_{2},I_{3}$ in the following two cases  respectively.

\textbf{Case 1.} $r^d u^\theta\in L^{p,q}_T$, $\frac{2}{p}+\frac{3}{q}\leq 1-d$, $0\leq d<1$, $ \frac{3}{1-d} < q\leq \infty$, $\frac{2}{1-d} \leq p\leq\infty$.\\

From  H\"{o}lder's inequality, Lemma \ref{lem2.1} ($s=\frac{2dq}{q-2}$, $q_*=\frac{2q}{q-2}$) and (\ref{2.6})-(\ref{2.6-1}), we get, for all $t\in[0,T^*)$,
\begin{eqnarray}
I_{1}
&\leq&\|r^{d}u^{\theta}\|_{q} ~\| \frac{\partial_{r}\frac{u^{r}}{r}}{r^{d}}~\|_{\frac{2q}{q-2}} ~\|\partial_{3}\Phi\|_{2}\nonumber\\
&\leq&C\|r^{d}u^{\theta}\|_{q} ~\| \tilde\nabla \partial_{r}\frac{u^{r}}{r}\|^{\frac{3}{q}+d}_{2}~ \|\partial_{r}\frac{u^{r}}{r}\|^{1-\frac{3}{q}-d}_{2}~ \|\partial_{3}\Phi\|_{2}\nonumber\\
&\leq&C\|r^{d}u^{\theta}\|_{q} ~\| \tilde\nabla \Gamma\|^{\frac{3}{q}+d}_{2}~ \|\Gamma\|^{1-\frac{3}{q}-d}_{2}~ \| \tilde\nabla\Phi\|_{2}\nonumber\\
&\leq&C \|r^{d}u^{\theta}\|_{q}^{\frac{2}{1-\frac{3}{q}-d}}~\|\Gamma\|_{2}^{2}+\frac{1}{8}\|\tilde\nabla\Phi\|_{2}^{2}+\frac{1}{8}\|\tilde\nabla\Gamma\|_{2}^{2}.\label{3.5}
\end{eqnarray}
Similarly, we have for all $t\in[0,T^*)$
\begin{equation}\label{3.6}
I_{2}\leq C \|r^{d}u^{\theta}\|_{q}^{\frac{2}{1-\frac{3}{q}-d}}~\|\Gamma\|_{2}^{2}+\frac{1}{8}\|\tilde\nabla\Phi\|_{2}^{2}+\frac{1}{8}\|\tilde\nabla\Gamma\|_{2}^{2}.
\end{equation}
From H\"{o}lder's inequality, Lemma \ref{lem2.1} ($s=\frac{(1+d)q}{q-1}$, $q_*=\frac{2q}{q-1}$) and (\ref{2.6})-(\ref{2.6-1}), we obtain for all $t\in[0,T^*)$
\begin{eqnarray}\label{3.7}
I_{3}
&\leq& \|r^{d}u^{\theta}\|_{q} \|r^{-\frac{1+d}{2}}\Gamma \|_{\frac{2q}{q-1}} \|r^{-\frac{1+d}{2}}\Phi \|_{\frac{2q}{q-1}} \nonumber\\
&\leq& C  \|r^{d}u^{\theta}\|_{q}(\|\Gamma\|_2\|\Phi\|_2)^{\frac{1-\frac{3}{q}-d}{2}}(\|\tilde\nabla\Gamma\|_{2}\|\tilde\nabla\Phi\|_{2})^{\frac{1+\frac{3}{q}+d}{2}}\nonumber\\
& \leq&  C \|r^{d}u^{\theta}\|_{q}^{\frac{2}{1-\frac{3}{q}-d}}\|\Gamma\|_{2}\|\Phi\|_2+\frac{1}{8}\|\tilde\nabla\Phi\|_{2}^{2}+\frac{1}{8}\|\tilde\nabla\Gamma\|_{2}^{2}.
\end{eqnarray}
From (\ref{3.4}), (\ref{3.5}), (\ref{3.6}), (\ref{3.7}), we get for all $t\in[0,T^*)$
\begin{equation}\label{3.8}
\frac{1}{2}\frac{d}{dt}(\|\Phi\|_{2}^{2}+\|\Gamma\|_{2}^{2})+\frac{1}{2}\|\tilde\nabla\Phi\|_{2}^{2}+\frac{1}{2}\|\tilde\nabla\Gamma\|_{2}^{2}\leq C (1+\|r^{d}u^{\theta}\|_{q}^{p})(\|\Phi\|_{2}^{2}+\|\Gamma\|_{2}^{2}).
\end{equation}
Using Gronwall's inequality, we have
\begin{equation}\label{3.9}
\sup_{t\in[0,T^*)} \|\Gamma\|_{L^{\infty,2}_{t}}^2\leq (\|\Phi_0\|_{2}^{2}+\|\Gamma_0\|_{2}^{2})\exp(CT+C\|r^{d}u^{\theta}\|_{L^{p,q}_{T}}^p)<\infty.
\end{equation}

\textbf{Case 2.} $r^d u^\theta\in L^{\infty,\frac{3}{1-d}}_T$, $0\leq d<1$.
For a small constant $\varepsilon>0$ given in (\ref{*}), there exists $\alpha>0$ such that
    \begin{equation}
      \|r^d u^\theta 1_{r\leq \alpha}\|_{L^{\infty,\frac{3}{1-d}}_T}\leq \varepsilon.
    \end{equation}
Similar to (\ref{3.5}), we have for all $t\in[0,T^*)$
    \begin{eqnarray}
I_{1}
&\leq&C \|1_{r\leq \alpha }r^{d}u^{\theta}\|_{\frac{3}{1-d}}\|\tilde\nabla\Phi\|_{2}\|\tilde\nabla\Gamma\|_{2}
+C\|r^{d}u^{\theta}\|_{\frac{3}{1-d}}\|1_{r\geq\alpha} \frac{\partial_{r}\frac{u^{r}}{r}}{r^{d}}~\|_{\frac{6}{1+2d}} ~\|\partial_{3}\Phi\|_{2}\nonumber\\
    &\leq&C \|1_{r\leq \alpha }r^{d}u^{\theta}\|_{\frac{3}{1-d}}\|\tilde\nabla\Phi\|_{2}\|\tilde\nabla\Gamma\|_{2}
\nonumber\\
&+&C_\alpha\|r^{d}u^{\theta}\|_{\frac{3}{1-d}}(\| \partial_{r}\frac{u^{r}}{r}\|_{2}+
\| \partial_{r}\frac{u^{r}}{r}  \|_{2}^{\frac{1+2d}{3}}\|\tilde{\nabla} \partial_{r}\frac{u^{r}}{r} \|_{2}^{\frac{2-2d}{3}}
) \|\partial_{3}\Phi\|_{2}\nonumber\\
&\leq&C \|1_{r\leq \alpha }r^{d}u^{\theta}\|_{\frac{3}{1-d}}\|\tilde\nabla\Phi\|_{2}\|\tilde\nabla\Gamma\|_{2}
+C_\alpha\|r^{d}u^{\theta}\|_{\frac{3}{1-d}}(\|\Gamma\|_{2}+
\|\Gamma\|_{2}^{\frac{1+2d}{3}}\|\tilde{\nabla}\Gamma\|_{2}^{\frac{2-2d}{3}}
) \|\tilde{\nabla}\Phi\|_{2}\nonumber\\
&\leq&(\frac{1}{8}+C \varepsilon)(\|\tilde\nabla\Phi\|_{2}^2+\|\tilde\nabla\Gamma\|_{2}^2)
+C_\alpha(1+\|r^{d}u^{\theta}\|_{\frac{3}{1-d}}^{\frac{6}{1+2d}})\|\Gamma\|_{2}^2,\label{3.11-0}
\end{eqnarray}
where we use the following estimate for $p\in[2,\infty)$,
    \begin{eqnarray*}
      &&\left(\int_{\R}\int^\infty_\alpha |f|^prdrdx_3
      \right)^{\frac{1}{p}}\\
      &\leq&C\left(\int_{\R}\int^\infty_\alpha |f|^2r^{\frac{2}{p}}drdx_3
      \right)^{\frac{1}{p}}\left(\int_{\R}\int^\infty_\alpha |\tilde{\nabla}(f r^{\frac{1}{p}})|^2drdx_3
      \right)^{\frac{1}{2}-\frac{1}{p}}\\
      &\leq&C_\alpha \|f\|_2+C_\alpha\|f\|_2^\frac{2}{p}\|\tilde{\nabla}f\|_2^{1-\frac{2}{p}}.
    \end{eqnarray*}
Similarly, we get for all $t\in[0,T^*)$
    \begin{equation}
      I_2\leq (\frac{1}{8}+C \varepsilon)(\|\tilde\nabla\Phi\|_{2}^2+\|\tilde\nabla\Gamma\|_{2}^2)
+C_\alpha(1+\|r^{d}u^{\theta}\|_{\frac{3}{1-d}}^{\frac{6}{1+2d}})\|\Gamma\|_{2}^2,\label{3.12-0}
    \end{equation}
    and
    \begin{eqnarray}
I_{3}
&\leq&C \|1_{r\leq \alpha }r^{d}u^{\theta}\|_{\frac{3}{1-d}}\|\tilde\nabla\Phi\|_{2}\|\tilde\nabla\Gamma\|_{2}
+C\|r^{d}u^{\theta}\|_{\frac{3}{1-d}}\|1_{r\geq\alpha} r^{-\frac{1+d}{2}}\Gamma\|_{\frac{6}{2+d}}\|1_{r\geq\alpha} r^{-\frac{1+d}{2}}\Phi\|_{\frac{6}{2+d}}\nonumber\\
    &\leq&C \|1_{r\leq \alpha }r^{d}u^{\theta}\|_{\frac{3}{1-d}}\|\tilde\nabla\Phi\|_{2}\|\tilde\nabla\Gamma\|_{2}\nonumber\\
        &&
+C_\alpha\|r^{d}u^{\theta}\|_{\frac{3}{1-d}}(\|\Gamma\|_{2}+
\|\Gamma\|_{2}^{\frac{2+d}{3}}\|\tilde{\nabla}\Gamma\|_{2}^{\frac{1-d}{3}}
) (\|\Phi\|_{2}+
\|\Phi\|_{2}^{\frac{2+d}{3}}\|\tilde{\nabla}\Phi\|_{2}^{\frac{1-d}{3}}
)\nonumber\\
&\leq&(\frac{1}{8}+C \varepsilon)(\|\tilde\nabla\Phi\|_{2}^2+\|\tilde\nabla\Gamma\|_{2}^2)
+C_\alpha(1+\|r^{d}u^{\theta}\|_{\frac{3}{1-d}}^{\frac{3}{2+d}})(\|\Phi\|_{2}^{2}+\|\Gamma\|_{2}^2).\label{3.13}
\end{eqnarray}
From (\ref{3.4}), (\ref{3.11-0}), (\ref{3.12-0}), (\ref{3.13}), we have, for all $t\in[0,T^*)$,
\begin{equation*}
\frac{1}{2}\frac{d}{dt}(\|\Phi\|_{2}^{2}+\|\Gamma\|_{2}^{2})+(\frac{5}{8}-3C\varepsilon)(\|\tilde\nabla\Phi\|_{2}^{2}+
\|\tilde\nabla\Gamma\|_{2}^{2})\leq C_\alpha(1+\|r^{d}u^{\theta}\|_{\frac{3}{1-d}}^{\frac{6}{1+2d}})
(\|\Phi\|_{2}^{2}+\|\Gamma\|_{2}^{2}).
\end{equation*}
When
    \begin{equation}
      3C\varepsilon=\frac{1}{8},\label{*}
    \end{equation}
    and applying Gronwall's inequality, we get
\begin{equation}
\sup_{t\in[0,T^*)} \|\Gamma\|_{L^{\infty,2}_{t}}^2\leq
(\|\Phi_0\|_{2}^{2}+\|\Gamma_0\|_{2}^{2})\exp\{C_\alpha T(1+\|r^{d}u^{\theta}\|_{L^{\infty,\frac{3}{1-d}}_{T}}^{\frac{6}{1+2d}})\}<\infty.\label{3.15}
\end{equation}
From Lemma \ref{lem2.6}, (\ref{3.9}) and (\ref{3.15}), then $\mathbf{u}$ can be continued beyond $T^*$, which contradicts with the definition of $T^*$. Thus, $T^*>T$,
 which finishes the proof of Theorem \ref{thm1.2}. $\hfill\Box$

Using Theorem \ref{thm1.2} and Lemma \ref{lem2.2}, we obtain Theorem \ref{thm1.1} as follows.

\vspace{0.5cm}
\noindent\textbf{Proof of Theorem \ref{thm1.1}.}

Choosing $\sigma=2q-1>1$, $F=ru^{\theta}$ and $f=r\omega^{3}$ in  Lemma \ref{lem2.2}, we deduce that
$$
\int_{0}^{\infty}(\frac{u^{\theta}}{r})^{q}~rdr<C(q)^{q}\int_{0}^{\infty}(\omega^{3})^{q}~rdr,
$$
and
$$
\|\frac{u^{\theta}}{r}\|_{q}<C(q)\|\omega^{3}\|_{q}.
$$
Using the fact that $ru^{\theta}\in L^{\infty,\infty}_{T}$ and the interpolation theorem, we obtain that
$
u^{\theta}\in L^{2p,2q}_{T}.
$ Then $\mathbf{u}$ is regular, by Theorem \ref{thm1.2}.   $\hfill\Box$

\vspace{0.5cm}
\noindent\textbf{Proof of Theorem \ref{thm1.5}}

Assume that $T^*\leq T$.
Applying Lemma \ref{lem2.3}, the potential function $\psi$ satisfies
\begin{equation}\label{3.16}
u^{r}=-\partial_{3}\psi,\ \
ru^{3}=\partial_{r}(r\psi).
\end{equation}
From (\ref{3.16}), we get
\begin{equation}\label{3.17}
\frac{\psi}{r^{1-d}}=\int_{0}^{1}s^{1-d}(sr)^{d}u^{3}(sr,x_{3})~ds.
\end{equation}
Applying Minkowski's inequality, we obtain
\begin{equation}\label{3.18}
\|\frac{\psi}{r^{1-d}}\|_{q}\leq C \|r^{d}u^{3}\|_{q}.
\end{equation}

\textbf{Case 1.} $r^du^3\in L^{p,q}_T$, $\frac{2}{p}+\frac{3}{q}\leq1-d, 0\leq d<1, \frac{3}{1-d} < q\leq\infty,\frac{2}{1-d} \leq p<\infty$.

Pick $0<k=1-\frac{3}{q(1-d)}\leq1$.
Multiplying the $u^{\theta}$ equation of (\ref{1.2})$_2$ by $(u^{\theta})^{3}$, and integrating the resulting equation over $\R^{3}$,  using the integration by parts,
 H\"{o}lder's inequality and the Sobolev embedding theorem, we have
\begin{eqnarray}
&&\frac{1}{4}\frac{d}{dt}\|u^{\theta}\|_{4}^{4}+\frac{3}{4}\|\tilde\nabla(u^{\theta})^{2}\|_{2}^{2}+\|\frac{(u^{\theta})^{2}}{r}\|_{2}^{2}\nonumber\\
        &=&-\int_{\R^{3}}\frac{u^{r}}{r}(u^{\theta})^{2}(u^{\theta})^{2}~dx\nonumber\\
        &=&\int_{\R^{3}}\partial_{3}\frac{\psi}{r}(u^{\theta})^{2}(u^{\theta})^{2}~dx\nonumber\\
&=&-2\int_{\R^{3}}\frac{\psi}{r}(u^{\theta})^{2}\partial_{3}(u^{\theta})^{2}~dx\nonumber \\
&=&-2\int_{\R^{3}}\frac{\psi}{r^{1-d}}(\frac{(u^{\theta})^{2}}{r})^{d}(u^{\theta})^{2(1-d)}\partial_{3}(u^{\theta})^{2}~dx\nonumber\\
&\leq&C \|\frac{\psi}{r^{1-d}}\|_{q}~\|\frac{(u^{\theta})^{2}}{r}~\|_{2}^{d}~\|(u^{\theta})^{2}\|_{2}^{(1-d)k}
\|(u^{\theta})^{2}\|_{6}^{(1-d)(1-k)}~\|\partial_{3}(u^{\theta})^{2}\|_{2} \nonumber\\
&\leq&C \|r^{d}u^{3}\|_{q}~\|\frac{(u^{\theta})^{2}}{r}~\|_{2}^{d}~\|(u^{\theta})^{2}\|_{2}^{(1-d)k}
\|\tilde\nabla(u^{\theta})^{2}\|_{2}^{1+(1-d)(1-k)}\nonumber \\
&\leq&C  \|r^{d}u^{3}\|_{q}^{\frac{2}{1-d-\frac{3}{q}}}~\|(u^{\theta})^{2}\|_{2}^{2}+\frac{1}{2}\|\frac{(u^{\theta})^{2}}{r}\|_{2}^{2}
+\frac{1}{2}\|\tilde\nabla(u^{\theta})^{2}\|_{2}^{2},\label{3.19}
\end{eqnarray}
for all $t\in[0,T^*)$. Applying Gronwall's inequality, we obtain
\begin{equation}
\sup_{t\in[0,T^*)}\|u^{\theta}(t)\|^{4}_{4} \leq \|u^{\theta}_0\|^{4}_{4}~ \exp (CT+C\|r^{d}u^{3}\|_{L^{p,q}_{T}}^p) < \infty.
\end{equation}

\textbf{Case 2.}  $r^du^3\in L^{\infty,\frac{3}{1-d}}_T$,  $0\leq d<1$. For a small constant $\varepsilon_1>0$ given in (\ref{**}), there exists $\beta>0$ such that
    \begin{equation}
      \|r^du^31_{r\leq \beta}\|_{L^{\infty,\frac{3}{1-d}}_T}\leq \varepsilon_1.
    \end{equation}

Using the similar arguments as that in the proof of (\ref{3.19}), we have for all $t\in[0,T^*)$
\begin{eqnarray}
&&\frac{1}{4}\frac{d}{dt}\|u^{\theta}\|_{4}^{4}+\frac{3}{4}\|\tilde\nabla(u^{\theta})^{2}\|_{2}^{2}+\|\frac{(u^{\theta})^{2}}{r}\|_{2}^{2}\nonumber\\
        &\leq&C \|r^{d}u^{3}1_{r\leq \beta}\|_{\frac{3}{1-d}}~\|\frac{(u^{\theta})^{2}}{r}\|_2^d~
\|\tilde\nabla(u^{\theta})^{2}\|_{2}^{2-d}
+C\|\frac{\psi}{r^{1-d}}\|_{\frac{3}{1-d}}\|1_{r\geq \beta}\frac{(u^{\theta})^{2}}{r^d}\|_{\frac{6}{1+2d}}\|\partial_{3}(u^{\theta})^{2}\|_2
\nonumber \\
&\leq&C \|r^{d}u^{3}1_{r\leq \beta}\|_{\frac{3}{1-d}}~\|\frac{(u^{\theta})^{2}}{r}\|_2^d~
\|\tilde\nabla(u^{\theta})^{2}\|_{2}^{2-d}\nonumber\\
    &&
+C_\beta\|r^{d}u^{3}\|_{\frac{3}{1-d}}(\|(u^{\theta})^{2}\|_{2}+\|(u^{\theta})^{2}\|_{2}^{\frac{1+2d}{3}}
\|\tilde{\nabla}(u^{\theta})^{2}\|_{2}^{\frac{2-2d}{3}})\|\partial_{3}(u^{\theta})^{2}\|_2
\nonumber \\
    &\leq&(\frac{1}{8}+C\varepsilon_1)(\|\tilde\nabla(u^{\theta})^{2}\|_{2}^{2}+\|\frac{(u^{\theta})^{2}}{r}\|_{2}^{2})
    +C_\beta(1+\|r^{d}u^{3}\|_{\frac{3}{1-d}}^\frac{6}{1+2d})\|(u^{\theta})^{2}\|_{2}^2.
\end{eqnarray}
Since
    \begin{equation}
      C\varepsilon_1=\frac{1}{8},\label{**}
    \end{equation}
and using Gronwall's inequality, we get
    \begin{equation}
\sup_{t\in[0,T^*)}\|u^{\theta}(t)\|^{4}_{4} \leq \|u^{\theta}_0\|^{4}_{4}~ \exp \{C_\beta T(1+\|r^{d}u^{3}\|_{L^{\infty,\frac{3}{1-d}}_{T}}^\frac{6}{1+2d})\} < \infty.
\end{equation}

Using the similar argument as that in the proof of Theorem \ref{thm1.2}, we have that $\mathbf{u}$ can be continued beyond $T^*$, which contradicts with the definition of $T^*$. Thus, $T^*>T$,
 which finishes Theorem \ref{thm1.5}. $\hfill\Box$

\vspace{0.5cm}

\noindent\textbf{Proof of Corollary \ref{cor1.6}.}\\

By (\ref{3.17}) and $r u^3\in L^{\infty,\infty}_T$, it is easy to show that
$$
\psi\in L^{\infty,\infty}_{T}.$$
Therefore
$$
\mathbf{b}=\nabla\times(\psi\boldsymbol{e_{\theta}}),~~~~\text{while}~~~
\psi\boldsymbol{e_{\theta}} \in L^{\infty,\infty}_{T}\hookrightarrow L^{\infty}((0,T);BMO).$$
From Theorem 1.4 in \cite{Z.Lei}, we get that  $\mathbf{u}$ is regular in $(0,T]\times \R^{3}$.  $\hfill\Box$

\section{Global well-posedness}\label{S4}

In this section, we are going to prove Theorem \ref{thm1.4} by the classical continuity method.
\begin{lem}\label{L4.1}
Let $\mathbf{u}\in C([0,T);H^2(\R^3))\cap L^2_{loc}([0,T);H^3(\R^3))$ be the unique axisymmetric solution of the Navier-Stokes equations with the axisymmetric initial data $\mathbf{u_{0}}\in H^{2}(\R^{3})$ and $\mathrm{div}~\mathbf{u_{0}}=0$. Assume $\|r^du^{\theta}\|_{L^{\infty}([0,t);L^{\frac{3}{1-d}}(\R^{3}))}\leq\frac{1}{C_{0}}$, where $C_{0}$ is a constant in (\ref{4.2}) and $t<T^{*}$. Then, we have
\begin{equation}\label{4.1}
\|\Phi(t)\|_{2}^{2}+\|\Gamma(t)\|_{2}^{2}\leq\|\Phi_{0}\|_{2}^{2}+\|\Gamma_{0}\|_{2}^{2}.
\end{equation}
\end{lem}
\begin{proof}
From the Sobolev-Hardy inequality (\ref{2.7-0}) and (\ref{3.4}), we have \textit{a priori} estimate on $(0,t)$,  that
\begin{eqnarray}
 \frac{1}{2}\frac{d}{dt}(\|\Phi\|_{2}^{2}+\|\Gamma\|_{2}^{2})+\|\tilde\nabla\Phi\|_{2}^{2}+\|\tilde\nabla\Gamma\|_{2}^{2}
&\leq& I_{1}+I_{2}+2I_{3}\nonumber\\
&\leq&C_{0} \|r^du^{\theta}\|_{\frac{3}{1-d}}\|\tilde\nabla\Gamma\|_{2}\|\tilde\nabla\Phi\|_{2}\label{4.2}\\
&\leq& \frac{1}{2}\|\tilde\nabla\Gamma\|_{2}^{2}+\frac{1}{2}\|\tilde\nabla\Phi\|_{2}^{2},\nonumber
\end{eqnarray}
which implies (\ref{4.1}).
\end{proof}

\begin{lem}\label{L4.2}
Under the conditions in Lemma \ref{L4.1}, then  there exists a positive constant $C_{1}$, such that,
\begin{equation}\label{4.3}
\|\omega^{\theta}\|^{2}_{L^{\infty}((0,t);L^{2}(\R^{3}))}\leq \|\omega_{0}^{\theta}\|_{2}^{2}+C_{1}(\|\Gamma_{0}\|_{2}+\|\Phi_{0}\|_{2})^{\frac{4}{3}}\|\mathbf{u_{0}}\|_{2}^{2}.
\end{equation}
\end{lem}
\begin{proof}
Multiplying the equation (\ref{1.3})$_2$ by $\omega^{\theta}$, and integrating the resulting equations over $\R^{3}$ respectively, using the integration by parts,
 H\"{o}lder's inequality and the Sobolev embedding Theorem, we have
\begin{eqnarray}
 \frac{1}{2}\frac{d}{dt}\|\omega^{\theta}\|_{2}^{2}+(\|\tilde\nabla\omega^{\theta}\|_{2}^{2}+\|\Gamma\|_{2}^{2})
&=&\int_{\R^{3}}(\Gamma u^{r}-2\Phi u^{\theta})\omega^{\theta}~dx\nonumber\\
&\leq&2(\|\Gamma\|_{2}\|u^{r}\|_{6}+\|\Phi\|_{2}\|u^{\theta}\|_{6})\|\omega^{\theta}\|_{3}\nonumber\\
&\leq& C (\|\Gamma\|_{2}+\|\Phi\|_{2})\|\nabla\mathbf{u}\|_{2}\|\omega^{\theta}\|_{2}^{\frac{1}{2}}\|\tilde\nabla\omega^{\theta}\|_{2}^{\frac{1}{2}}\nonumber\\
&\leq& C(\|\Gamma_{0}\|_{2}+\|\Phi_{0}\|_{2})^{\frac{4}{3}}\|\nabla\mathbf{u}\|_{2}^{\frac{4}{3}}\|\omega^{\theta}\|_{2}^{\frac{2}{3}}+\frac{1}{2}\|\tilde\nabla\omega^{\theta}\|_{2}^{2}
\nonumber\\
&\leq& C(\|\Gamma_{0}\|_{2}+\|\Phi_{0}\|_{2})^{\frac{4}{3}} \|\nabla\mathbf{u}\|_{2}^{2}+\frac{1}{2}\|\tilde\nabla\omega^{\theta}\|_{2}^{2}.\label{4.4}
\end{eqnarray}
From (\ref{3.1}) and (\ref{4.4}), we get (\ref{4.4}).
\end{proof}

\begin{lem}\label{L4.3}
Under the conditions in Lemma \ref{L4.1}, then  $$\|r^du^{\theta}\|_{L^{\infty}([0,t);L^{\frac{3}{1-d}}(\R^{3}))}\leq\frac{1}{2C_{0}}$$ if $u_{0}^{\theta}$ satisfies (\ref{1.7}).
\end{lem}
\begin{proof}
From (\ref{1.2})$_2$, using the integration by parts,
 H\"{o}lder's inequality and the Sobolev embedding Theorem, we have
\begin{eqnarray}
&&\frac{1-d}{3}\frac{d}{dt}\|r^du^{\theta}\|_{\frac{3}{1-d}}^{\frac{3}{1-d}}+\frac{4(1-d)(2+d)}{9}\|\nabla(r^{\frac{3d}{2(1-d)}}|u^{\theta}|^{\frac{3}{2(1-d)}})
\|_{2}^{2}\nonumber\\
&&+\frac{1-4d+(2+d)d^2}{1-d}\|r^{\frac{5d-2}{1-d}} (u^{\theta})^{\frac{3}{1-d}}\|_{1}\nonumber\\
&=&-\frac{1-3d}{1-d}\int_{\R^{3}} \frac{u^{r}}{r}~r^{\frac{3s}{1-d}}|u^{\theta}|^{\frac{3}{1-d}}dx\nonumber\\
&\leq&C\|\frac{u^{r}}{r}\|_{2}~\|r^{\frac{3d}{2(1-d)}}|u^{\theta}|^{\frac{3}{2(1-d)}}\|_{2}^{\frac{1}{2}}~\|r^{\frac{3d}{2(1-d)}}|u^{\theta}|^{\frac{3}{2(1-d)}}\|_{6}^{\frac{3}{2}}\nonumber\\
&\leq&C\|\omega^{\theta}\|_{2}^{4}\|r^du^{\theta}\|_{\frac{3}{1-d}}^{\frac{3}{1-d}}+\frac{(1-d)(2+d)}{9}\|\nabla(r^{\frac{3d}{2(1-d)}}|u^{\theta}|^{\frac{3}{2(1-d)}})
\|_{2}^{2}.
\end{eqnarray}
 Applying Gronwall's inequality, Lemma \ref{L4.2} and (\ref{1.7}), we get
\begin{eqnarray}
\|r^du^{\theta}\|_{\frac{3}{1-d}}&\leq&\|r^du^{\theta}_0\|_{\frac{3}{1-d}}\exp(C\int_{0}^{t}\|\omega^{\theta}\|_{2}^{4}~ds)\nonumber\\
&\leq&\|r^du^{\theta}_0\|_{\frac{3}{1-d}}\exp(C\|\omega^{\theta}\|^{2}_{L^{\infty}((0,t);L^{2}(\R^{3}))}   \int_{0}^{t}\|\omega^{\theta}\|_{2}^{2}~ds)\nonumber\\
&\leq&\|r^du^{\theta}_0\|_{\frac{3}{1-d}}\exp\left(C_{2}\|\mathbf{u_{0}}\|_{2}^{2}~(\|\omega_{0}^{\theta}\|_{2}^{2}+C_{1}(\|\Gamma_{0}\|_{2}+\|\Phi_{0}\|_{2})^{\frac{4}{3}}\|\mathbf{u_{0}}\|_{2}^{2})
\right)\nonumber\\
&\leq&\frac{1}{2C_{0}}.
\end{eqnarray}
\end{proof}

\noindent\textbf{Proof of Theorem \ref{thm1.4}.}

Assume that $T^*<\infty$.
From Lemma \ref{L4.3}, using the  classical continuity
method, we have $\|r^du^{\theta}\|_{L^{\infty}([0,T^*);L^{\frac{3}{1-d}}(\R^{3}))}\leq\frac{1}{C_{0}}$. Then, from  Lemma \ref{L4.1}, we have
$$
\sup_{t\in[0,T^*)}\|\Gamma(t)\|_{2}^{2}\leq\|\Phi_{0}\|_{2}^{2}+\|\Gamma_{0}\|_{2}^{2}.
$$
Applying Lemma \ref{lem2.6}, we have that  $\mathbf{u}$ can be continued beyond $T^*$,  which contradicts with the definition of $T^*$. Thus, $T^*=\infty$ and Theorem \ref{thm1.4} holds.  $\hfill\Box$

\section*{Acknowledgements}
  This work is
  partially supported by  NSF of
China under Grants 11271322,  11331005 and 11271017, National Program for
Special Support of Top-Notch Young Professionals.

\end{document}